\documentclass[11pt]{amsart}
\usepackage[dvipdfm]{graphicx, color}
\usepackage{amssymb, amsmath,amsthm}
\newtheorem{thm}{{\bf Theorem}}[section]
\newtheorem{lem}[thm]{{\bf Lemma}}
\newtheorem{cor}[thm]{{\bf Corollary}}
\newtheorem{prop}[thm]{{\bf Proposition}}

\newtheorem{rem}[thm]{Remark}

\newtheorem{ques}[thm]{Question}

\numberwithin{equation}{section}

\begin{document} 

\title[The boundary of a fibered face of the magic $3$-manifold]{ 
The boundary of a fibered face of the magic $3$-manifold and 
the asymptotic behavior of the minimal pseudo-Anosovs dilatations
}

\author[E. Kin]{%
    Eiko Kin
}
\address{%
Department of Mathematics, Graduate School of Science \\
Osaka University \\
Toyonaka, Osaka 560-0043, JAPAN
}
\email{%
       kin@math.sci.osaka-u.ac.jp}

\author[M. Takasawa]{%
    Mitsuhiko Takasawa
}
\address{%
       Department of Mathematical and Computing Sciences \\
       Tokyo Institute of Technology \\
       Ohokayama, Meguro, Tokyo 152-8552 Japan
}
\email{%
        takasawa@is.titech.ac.jp
}



\date{%
	\today
	}

\thanks{%
	The first author was partially supported by 
	Grant-in-Aid for Young Scientists (B) (No. 24740039), MEXT, Japan.
	} 

\begin{abstract} 
Let $\delta_{g,n}$ be the minimal dilatation of pseudo-Anosovs defined on an orientable surface of genus $g$ with $n$ punctures. 
Tsai proved that for any fixed $g \ge 2$, 
the logarithm of the minimal dilatation 
$\log \delta_{g,n}$ is on the order of $\log n/n$. 
We prove that 
if $2g+1$ is relatively prime to $s$ or $s+1$  for each $0 \le  s \le g$, then 
$$\limsup_{n \to \infty} \frac{n \log \delta_{g,n}}{\log n} \le 2.$$ 
Our examples of pseudo-Anosovs $\phi$'s which provide the upper bound above 
have the following property: 
The mapping torus $M_{\phi}$ of $\phi$ is a single hyperbolic $3$-manifold $N$ called the magic manifold, or 
the fibration of  $M_{\phi}$ comes from a fibration of $N$ 
 by Dehn filling  cusps along the boundary slopes of a fiber. 
The  main tool in this paper is the boundary of a fibered face of $N$. 
\end{abstract} 
\maketitle

\section{Introduction}
\label{section_intro}

Let $\varSigma= \varSigma_{g,n}$ be an orientable surface of genus $g$ with $n$ punctures and 
$\mathrm{Mod}(\varSigma)$ the mapping class group of $\varSigma$. 
According to the work of Nielsen and Thurston, 
elements of $\mathrm{Mod}(\varSigma)$ are classified into three types: 
periodic, reducible, pseudo-Anosov, see \cite{Thurston2}. 
Pseudo-Anosov mapping classes have rich dynamical and geometric properties. 
The hyperbolization theorem by Thurston~\cite{Thurston3} 
relates the dynamics of pseudo-Anosovs  and the geometry of hyperbolic fibered $3$-manifolds. 
The theorem asserts that 
$\phi \in \mathrm{Mod}(\varSigma)$ is pseudo-Anosov if and only if 
the mapping torus $M_{\phi}$ of $\phi$ is a hyperbolic $3$-manifold with finite volume.

Each pseudo-Anosov element $\phi \in \mathrm{Mod}(\varSigma)$ has a representative $\Phi: \varSigma \rightarrow \varSigma$ 
called a pseudo-Anosov homeomorphism. 
Such a homeomorphism is equipped with a constant $\lambda= \lambda(\Phi)>1$ 
called the {\it dilatation} of $\Phi$. 
If we let $\mathrm{ent}(\Phi)$ be the {\it topological entropy} of $\Phi$, then 
the identity $\mathrm{ent}(\Phi)= \log \lambda(\Phi)$ holds, see \cite[Expos\'e~10]{FLP}. 
The dilatation  $\lambda$ of $\Phi$ does not depend on the choice of a pseudo-Anosov homeomorphism $\Phi \in \phi$, 
and hence the {\it dilatation} $\lambda(\phi)$ of $\phi$ is defined to be $\lambda(\Phi)$. 
We call the quantities $\mathrm{ent}(\phi)= \log \lambda(\phi)$ and 
$\mathrm{Ent}(\phi)= |\chi(\varSigma)| \, \log \lambda(\phi)$ 
the {\it entropy} and {\it normalized entropy} of  $\phi$ respectively, 
where $\chi(\varSigma)$ is the Euler characteristic of $\varSigma$.

If we fix  $\varSigma$, the set of entropies of pseudo-Anosovs defined on $\varSigma$ 
is a closed discrete subset of  ${\Bbb R}$, 
see \cite{Ivanov}. 
In particular there exists a minimal entropy, and hence there exists a minimal dilatation. 
We denote by  $\delta(\varSigma)>1$,  the minimal dilatation of 
pseudo-Anosov elements in  $\mathrm{Mod}(\varSigma)$. 
The minimal dilatations  are determined in only a few cases, see \cite{CH}. 

Let us set $\delta_{g,n}= \delta(\varSigma_{g,n})$ and $\delta_g= \delta_{g,0}$. 
Penner proved in \cite{Penner} 
\footnote{Let  $A_g$ and  $B_g$ be  functions  on $g$. 
We write $A_g \asymp B_g$ if there exists a constant $c$, independent of $g$, such that 
$A_g/c  < B_g < c A_g$.}that 
$\log \delta_g \asymp \frac{1}{g}$. 
This work by Penner was a starting point on the study of the asymptotic behavior of the minimal dilatations on surfaces varying topology. 
Later it was proved by Hironaka-Kin\cite{HK} that $\log \delta_{0,n} \asymp \frac{1}{n}$, and 
by Tsai\cite{Tsai} that $\log \delta_{1,n} \asymp \frac{1}{n}$. 
See also Valdivia\cite{Valdivia}. 
The following  theorem, due to Tsai, is in contrast with the cases of genus $0$ or $1$.

\begin{thm}[\cite{Tsai}] 
\label{thm_Tsai}
For any fixed $g \ge 2$, we have 
$$\log \delta_{g,n} \asymp \frac{\log n}{n}.$$
\end{thm}

\noindent
We ask the following question which is motivated by Theorem~\ref{thm_Tsai}.

\begin{ques}
\label{ques_analogy} 
Given $g \ge 2$, 
does $\displaystyle \lim_{n \to \infty} \frac{n \log \delta_{g,n}}{\log n}  $ exist? 
What is its value?
\end{ques}

\noindent
This is an analogous question, posed by McMullen, 
which is asking whether $\displaystyle \lim_{g \to \infty} g \log \delta_g$ exists or not, see \cite{McMullen}.

\begin{thm}
\label{thm_Asymptotic}
Given $g \ge 2$, there exists a sequence $\{n_i\}^{\infty}_{i=0}$ with $n_i \to \infty$   
such that 
$$\limsup_{i \to \infty} \frac{n_i \log \delta_{g,n_i}}{\log n_i} \le 2.$$ 
\end{thm}

\noindent
We note that for any $g \ge 2$, 
Tsai's examples in \cite{Tsai} yield the upper bound 
$\displaystyle\limsup_{n \to \infty} \tfrac{n \log \delta_{g,n}}{\log n} \le 2(2g+1)$, 
which is proved by a similar computation as in the proof of Theorem~\ref{thm_Asymptotic}.

We define the polynomial $B_{(g,p)}(t)$ for nonnegative integers $g$ and $p$: 
$$B_{(g,p)}(t)= t^{2p+1}(t^{2g+1}-1)+1-2t^{p+g+1}-t^{2g+1}. $$
We shall see that there exists a unique real root $r_{(g,p)}$ greater than $1$ of $B_{(g,p)}(t)$ 
such that 
$$\displaystyle\lim_{p \to \infty} \frac{p \log r_{(g,p)}}{ \log p} =1$$ 
(Lemma~\ref{lem_asymp_agp}). 
The root $r_{(g,p)}$ gives the following upper bound of the minimal dilatations.

\begin{thm}
\label{thm_UpperBound}
For $g \ge 2$ and $p \ge 0$, suppose that $\gcd(2g+1, p+g+1)=1$. 
Then 
$$\delta_{g,2p+i} \le r_{(g,p)} \hspace{4mm}\mbox{for\  each\ } i \in \{1,2,3,4\}. $$
\end{thm}

If $g$ enjoys $(*)$ in the next theorem~\ref{thm_main}, then 
one can take a subsequence $\{n_i\}_{i=0}^{\infty}$ in Theorem~\ref{thm_Asymptotic} 
to be the sequence $\{n\}_{n=1}^{\infty}$ of natural numbers.

\begin{thm}
\label{thm_main}
Suppose that $g \ge 2$  satisfies 
\begin{quote} 
$(*)$ \hspace{2mm}
$\gcd(2g+1, s)=1$ or $\gcd(2g+1, s+1)=1$ for each $0 \le  s \le g$.
\end{quote}
Then 
$$\limsup_{n \to \infty} \frac{n \log \delta_{g,n}}{\log n} \le 2.$$ 
\end{thm}

\noindent
For example, $(*)$ holds for $g=4$ since 
$9$ is relatively prime to $1,2,4$ and $5$; 
$(*)$ does not hold for $g=7$ because 
$\gcd(15,5) = 5 $ and $\gcd(15,6) =3$. 
We point out that 
infinitely many $g$'s satisfy $(*)$. 
In fact if $2g+1$ is prime, then 
$2g+1$ is relatively prime to $s'$ for each $1 \le s' \le g+1$. 
Such a  $g$ enjoys $(*)$, and this leads to 

\begin{cor}
If $2g+1$ is prime for $g \ge 2$, then 
$$\limsup_{n \to \infty} \frac{n \log \delta_{g,n}}{\log n} \le 2.$$ 
\end{cor}

\begin{rem}
One can simplify $(*)$ in Theorem~\ref{thm_main}, 
since $2g+1$ is relative prime to $1,2$ and $g$. 
In the case  $g \ge 5$, 
$(*)$  is equivalent to 
\begin{quote}
$(**)$ \hspace{2mm} 
$\gcd(2g+1, s)=1$ or $\gcd(2g+1, s+1)=1$ for each $3 \le  s \le g-2$. 
\end{quote}
\end{rem}

Our results are proved by using the theory  on fibered faces of hyperbolic fibered $3$-manifolds $M$, 
developed by Thurston\cite{Thurston1}, Fried\cite{Fried}, Matsumoto\cite{Matsumoto} 
and McMullen\cite{McMullen}. (See Section~\ref{section_Thurston}.) 
Let $\| \cdot \|: H_2(M, \partial M; {\Bbb R}) \rightarrow {\Bbb R} $ be the Thurston norm, and 
let $\Omega$ be a fibered face of $M$. 
The work of Thurston  tells us  that 
if $M$ has the second Betti number more than $1$, 
then it admits a family of fibrations on $M$ dominated by $int(C_{\Omega})$, 
where $C_{\Omega}$ is the core over $\Omega$ with the origin and $int(C_{\Omega})$ is its interior. 
In other words, such a fibered $3$-manifold  provides infinitely many pseudo-Anosovs 
defined on surfaces with variable topology. 
By work of Fried, 
the entropy function defined on these fibrations 
admits a unique continuous extension 
$\mathrm{ent}: int(C_{\Omega}) \rightarrow {\Bbb R}$. 
By the continuity of $\|\cdot\|$ and $\mathrm{ent}$, 
we have the continuous function 
$$\mathrm{Ent}= \|\cdot \| \, \mathrm{ent}(\cdot) : int(C_{\Omega}) \to \mathbb{R}.$$ 
The normalized entropy function  $\mathrm{Ent}$ is constant on each ray in $int(C_{\Omega})$ through the origin.  
It is shown by Fried that 
 the restriction $\mathrm{ent}|_{int(\Omega)} (= \mathrm{Ent}|_{int(\Omega)}): int(\Omega) \rightarrow {\Bbb R}$ 
has the property such that 
$\mathrm{ent}(a)$ goes to $\infty$ as $a \in int(\Omega)$ goes to a point on the boundary of $ \Omega$.

These properties give us the following observation: 
Fix a manifold $M$ as above. 
For any compact set $\mathcal{D} \subset int(\Omega)$, 
there exists a constant $C= C_{\mathcal D}>0$ satisfying the following. 
Let $a \in int(C_{\Omega})$ be any integral class  of $H_2(M, \partial M; {\Bbb Z}) $
and let $\Phi_a$ be  the monodromy of the fibration associated to $a$. 
Then the normalized entropy $\mathrm{Ent}(\Phi_a)$ is bounded by $C$ from above 
whenever $\overline{a} \in \mathcal{D}$, 
where $\overline{a}$ is the projective class of $a$. 

This observation enables us to investigate the asymptotic behavior of the minimal dilatations. 
The following asymptotic inequalities (which are the best known upper bounds) are proved by using a similar technique.

\begin{enumerate}
\item[(1)] 
$\displaystyle  \limsup_{n \to \infty} n \log \delta_{0,n} \le 2 \log (2+ \sqrt{3})$, see \cite{HK,KT}. 

\item[(2)] 
$\displaystyle  \limsup_{n \to \infty} n \log \delta_{1,n} \le 2 \log \lambda_0$, 
where 
$\lambda_0 \approx 2.2966$ is  the largest real root of $t^4 -2t^3 - 2t+1$, 
see \cite{KKT}. 

\item[(3)] 
 $\displaystyle \limsup_{g \to \infty} \, g \log  \delta_g \le \log(\tfrac{3+ \sqrt{5}}{2})$, 
see \cite{Hironaka,AD,KT1}. 
\end{enumerate}

\noindent
However for any fixed $g \ge 2$, 
the observation as above doesn't work 
to investigate the asymptotic behavior  $ \delta_{g,n}$ varying $n$ because of Theorem~\ref{thm_Tsai}. 
Theorem~\ref{thm_Tsai} implies that 
there exists no constant $C>0$, independent of $n$ so that 
$|\chi(\varSigma_{g,n})| \log \delta_{g,n}< C$. 
Thus if there exists a sequence of integral  classes $\{a_i\}$ 
with $a_i \in int(C_{\Omega}) \subset H_2(M, \partial M; {\Bbb Z})$  such that 
the fiber of the fibration associated to $a_i$ is a surface of genus $g$ having $n_i$ boundary components 
with $n_i \to \infty$, 
then the accumulation points of the sequence of projective classes $\{\overline{a}_i\}$ must lie on the boundary of $\Omega$.   
(This is because there exists no constant $C>0$, independent of $i$, such that 
$ \mathrm{Ent}(\Phi_{a_i}) (= |\chi(\varSigma_{g, n_i})| \log(\Phi_{a_i})) \le C$.)

Nevertheless we focus on a fibered face of a particular hyperbolic fibered $3$-manifold, called the {\it magic manifold} $N$. 
This manifold is the exterior of the $3$ chain link $\mathcal{C}_3$, see Figure~\ref{fig_poly}. 
Our examples of pseudo-Anosovs $\phi$'s which provide the upper bounds in 
Theorems~\ref{thm_Asymptotic}, \ref{thm_UpperBound} and  \ref{thm_main} 
have the following property: 
The mapping torus $M_{\phi}$ of $\phi$ is homeomorphic to $N$, or 
the fibration of  $M_{\phi}$ comes from a fibration of $N$ 
 by Dehn filling  cusps along the boundary slopes of a fiber. 
 We also point out that a family of the integral classes of $H_2(N, \partial N; {\Bbb Z})$  is a main ingredient 
 to prove the asymptotic inequalities (1)--(3) above, see \cite{KKT}.

\begin{figure}[htbp]
\begin{center}
\includegraphics[width=4.7in]{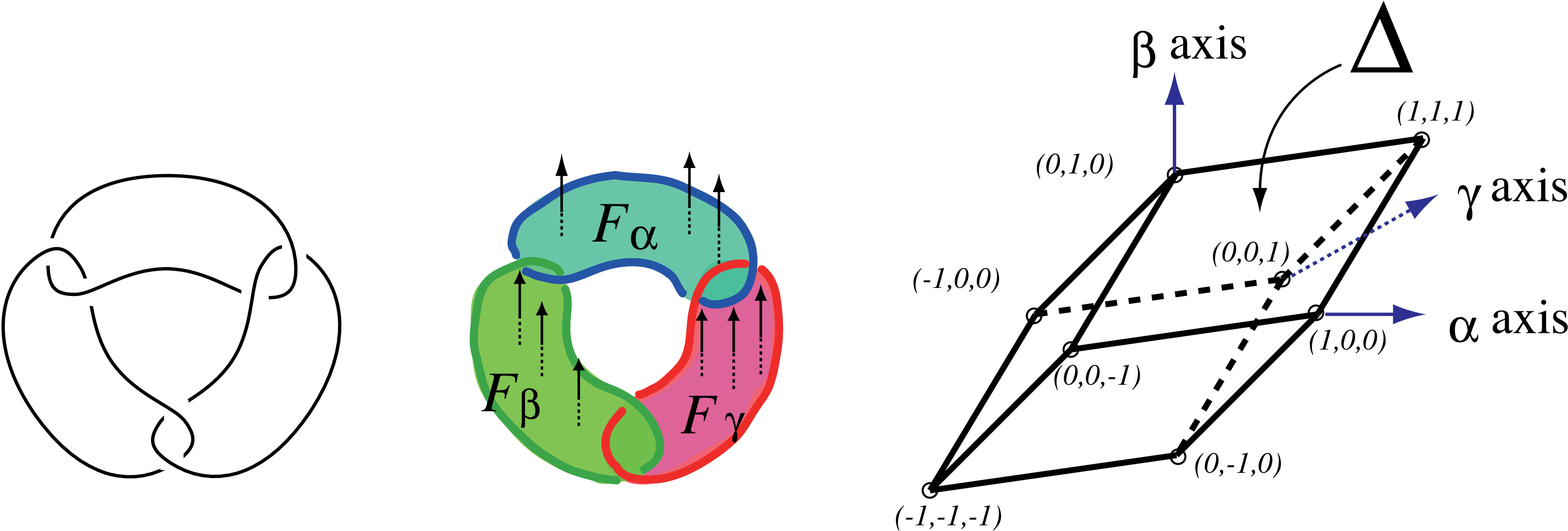}
\caption{(left) $3$ chain link $\mathcal{C}_3$. 
(center) $F_{\alpha}$, $F_{\beta}$, $F_{\gamma}$. 
[arrows indicate the normal direction of oriented surfaces.] 
(right) Thurston norm ball $U_N$. (fibered face $\Delta$ is indicated.)}
\label{fig_poly}
\end{center}
\end{figure}

We turn to the hyperbolic volume of hyperbolic $3$-manifolds. 
The set of volumes of hyperbolic $3$-manifolds is a well-ordered closed subset in ${\Bbb R}$ of order type $\omega^{\omega}$, 
see \cite{Thurston}. 
In particular if we fix a surface $\varSigma$, 
then there exists a minimum among volumes of hyperbolic $\varSigma$-bundles over the circle. 

The proofs of Theorems~\ref{thm_Asymptotic}, \ref{thm_main} imply the following.

\begin{prop}
\label{prop_volume}
Given $g \ge 2$, there exists a sequence $\{n_i\}^{\infty}_{i=0}$ with $n_i \to \infty$   
such that  the minimal volume of $\varSigma_{g,n_i}$-bundles over the circle is less than or equal to 
$\mathrm{vol}(N) \approx 5.3334$, the volume of the magic manifold $N$. 
Furthermore if $g \ge 2$  satisfies $(*)$, 
then for any $n \ge 3$, the minimal volume of $\varSigma_{g,n}$-bundles over the circle is less than or equal to $\mathrm{vol}(N)$. 
\end{prop}

We close the introduction by asking

\begin{ques}[cf. Theorems~\ref{thm_Asymptotic} and \ref{thm_main}]
Does $\displaystyle \limsup_{n \to \infty} \tfrac{n \log \delta_{g,n}}{\log n} \le 2$ hold 
for any $g \ge 2$? 
\end{ques}

\noindent
{\bf Acknowledments.} 
We would like to thank Eriko Hironaka for helpful conversations and comments. 

\section{Roots of polynomials} 
\label{section_roots}

This section concerns the asymptotic behavior of roots of families of polynomials. 
Let 
$$g(t)= a_n t^{b_n}+ a_{n-1} t^{b_{n-1}} + \cdots + a_1 t^{b_1} + a_0$$ 
be a polynomial with  real coefficients $a_0, a_1, \cdots, a_n$ 
($a_1, a_2, \cdots, a_n \ne 0$),  
where  
$g(t)$ is arranged in the order of descending powers of $t$. 
Let $\mathfrak{D}(g)$ be the number of variations in signs of the coefficients 
$a_n, a_{n-1}, \cdots, a_0$. 
For example if $g(t)= +t^4+t^3-2t^2+t-1$, then 
$\mathfrak{D}(g)=3$; 
if $h(t)=  +t^4+t^3-2t^2+t+1$, then 
$\mathfrak{D}(h)=2$. 
 Descartes's rule of signs (see \cite{Wang}) says that 
 the number of positive real roots of $g(t)$ 
(counted with multiplicities) is equal to either $\mathfrak{D}(g)$ 
 or less than $\mathfrak{D}(g)$  by an even integer.

\begin{lem}
\label{lem_asymptoticPoly}
Let $r \ge 0$, $s>0$ and $u > 0$ be integers. 
Let 
\begin{eqnarray*}
P_m(t)&=& t^{2m+r}(t^s-1)+1 - Q(t) t^m-t^u 
\\
&=& t^{2m+r+s}-t^{2m+r}- Q(t)t^m - t^u+1 
\end{eqnarray*}
be a polynomial for each $m \in {\Bbb N}$, 
where $Q(t)$ is a polynomial whose coefficients are positive integers. 
($Q(t)$ could be a positive constant.) 
\begin{enumerate}
\item[(1)] 
Suppose that  $t^{2m+r+s}$ is the leading term of $P_m(t)$. 
Then $P_m(t)$ has a unique real root $\lambda_m$ greater than $1$.

\item[(2)]
Given $0< c_1<1$ and  $c_2>1$, 
we have 
$$m^{\frac{c_1}{m}} < \lambda_m<  m^{\frac{c_2}{m}} \hspace{4mm} \mbox{for\ }\  m\ \mbox{large}. $$
In particular 
$$\lim_{m \to \infty} \frac{m \log \lambda_m}{\log m}=1.$$

\item[(3)] 
For any real numbers $q \ne 0$ and $v$, we have 
$$\displaystyle \lim_{m \to \infty} \frac{(qm+v) \log \lambda_{m}}{ \log (qm+v)}=q.$$

\end{enumerate}
\end{lem}

\begin{proof}
(1) 
Under the assumption on  $P_m(t)$,  
we have $\mathfrak{D}(P_m)=2$. 
By Descartes's rule of signs, the number of positive real roots of $P_m(t)$ is either $2$ or $0$. 
Since $P_m(0)=1$ and $P_m(1)= -Q(1)<0$, 
the number of positive real roots of $P_m(t)$ is exactly $2$.  
Because $P_m(t)$ goes to $\infty$ as $t$ does, 
$P_m(t)$ has a unique real root $\lambda_m>1$. 
\medskip

\noindent
(2) 
We have 
$$P_m(t) t^{-(2m+r)} = t^s-1 + t^{- (2m+r)} - Q(t) t^{-(m+r)} - t^{-(2m+r-u)}.$$
We define $f_m(t)$ and $g_m(t)$ such that $P_m(t) t^{-(2m+r)} = f_m(t)+ g_m(t)$ as follows. 
\begin{eqnarray*}
f_m(t)&=&  t^s-1 + t^{- (2m+r)}, \ \mbox{and}\ 
\\
g_m(t)&=& Q(t) t^{-(m+r)} + t^{-(2m+r-u)}.
\end{eqnarray*}
We let $t= m^{\frac{c}{m}}$ for $c>0$.  
Then 
\begin{eqnarray*}
f_m(m^{\frac{c}{m}}) 
&=&  (m^{\frac{c}{m}})^s-1 + (m^{\frac{c}{m}})^{-(2m+r)} 
\\
&=& 
\big((e^{\log m})^{\frac{c}{m}}\bigr)^s -1 + m^{-c(2+ \frac{r}{m})} 
\\
&=& 
e^{\frac{sc \log m}{m}} -1 + m^{-c(2+ \frac{r}{m})}. 
\end{eqnarray*}
By Maclaurin expansion of $e^{\frac{sc \log m}{m}}$, we have 
$$ e^{\frac{sc \log m}{m}} = 1 + \frac{sc \log m}{m}+ R_2, $$
where 
$$R_2 = \frac{e^w}{2} (\frac{sc \log m}{m})^2\ \mbox{for\ some\ } 0 < w < \frac{sc \log m}{m}. $$ 
Since $\frac{sc \log m}{m}$ goes to $0$ as $m$ goes to $\infty$, 
we may assume that $ \frac{e^w}{2} < B$ for some constant $B>0$. 
Then 

\begin{eqnarray*}
f_m(m^{\frac{c}{m}}) 
&=& \frac{sc \log m}{m}+ R_2 +\frac{m^{1-c(2+ \frac{r}{m})}}{m}
\\
&<&  \frac{sc \log m}{m} + B(\frac{sc \log m}{m})^2+ \frac{m^{1-c(2+ \frac{r}{m})}}{m} 
\\
&=&  \frac{sc \log m}{m} + Bs^2c^2 (\frac{\log m}{m})^2+  \frac{m^{1-c(2+ \frac{r}{m})}}{m} 
\\
&<&  \frac{sc \log m}{m} + Bs^2c^2 (\frac{\log m}{m}) +  \frac{m^{1-c(2+ \frac{r}{m})}}{m} 
\\
&=& \frac{(sc+Bs^2c^2)\log m+ m^{1-c(2+ \frac{r}{m})}}{m}. 
\end{eqnarray*}
(The last inequality comes from $0<\tfrac{\log m}{m}< 1$ for $m$ large.)
Thus 
\begin{equation}
\label{equation_f_m}
f_m(m^{\frac{c}{m}}) < \frac{(sc+Bs^2c^2)\log m+ m^{1-c(2+ \frac{r}{m})}}{m}. 
\end{equation}
The first  equality 
$f_m(m^{\frac{c}{m}}) = \tfrac{sc \log m}{m}+ R_2 +\frac{m^{1-c(2+ \frac{r}{m})}}{m}$ 
above together with  $R_2>0$ and $\frac{m^{1-c(2+ \frac{r}{m})}}{m}>0$ 
tells us that 
\begin{equation}
\label{equation_second}
f_m(m^{\frac{c}{m}}) > \frac{sc \log m}{m}. 
\end{equation}

Recall that all coefficients of $Q(t)$ (appeared in $P_m(t)$) are positive integers. 
If we write $Q(t)= \displaystyle\sum_{j=0}^{\ell} a_j t^j$, 
where $a_j \ge 0$, 
then 
\begin{eqnarray*}
g_m(m^{\frac{c}{m}}) 
&=& 
Q(m^{\frac{c}{m}})  m^{-c(1+ \frac{r}{m})} + m^{-c(2+ \frac{r}{m}-\frac{u}{m})}
\\
&=& 
\Bigl(\sum_{j=0}^{\ell} a_j m^{-c(1+ \frac{r}{m}- \frac{j}{m})}\Bigr)+ m^{-c(2+ \frac{r}{m}-\frac{u}{m})}. 
\end{eqnarray*} 
Thus we obtain 
\begin{equation}
\label{equation_g_m} 
g_m(m^{\frac{c}{m}}) =   \frac{\bigl(\sum_{j=0}^{\ell} a_{j}   m^{1 -c(1+ \frac{r}{m}- \frac{j}{m})}\bigr) + m^{1-c(2+ \frac{r}{m}-\frac{u}{m})}}{m}.
\end{equation}
For the proof of the claim (1), 
it is enough to prove that 
for $0 < c_1< 1$ and  $c_2>1$, 
we have 
$f_m(m^{\frac{c_1}{m}})<g_m(m^{\frac{c_1}{m}}) $  and 
$f_m(m^{\frac{c_2}{m}})>g_m(m^{\frac{c_2}{m}}) $ for $m $ large. 

First, suppose that $0< c < \frac{1}{2}$. 
Let us consider how the following four terms grow. 
\begin{equation}
\label{equation_four}
\log m,\ m^{1-c(2+ \frac{r}{m})},\ m^{1 -c(1+ \frac{r}{m}- \frac{j}{m})}\ \mbox{and}\ m^{1-c(2+ \frac{r}{m}-\frac{u}{m})}.
\end{equation}
The first two terms are appeared in  (\ref{equation_f_m}), and 
the last two are coming from (\ref{equation_g_m}). 
All four terms go to $\infty$ as $m$ does, 
since the last three terms have the positive powers of $m$. 
Note that for any $C>0$, 
we have 
$\log m < m^C$ for $m$ large.  
Keeping in mind of this, 
we observe that among the four terms in (\ref{equation_four}), 
$m^{1 -c(1+ \frac{r}{m}- \frac{j}{m})}$ is dominant. 
This is because
$$1 -c(1+ \frac{r}{m}- \frac{j}{m})> 1-c(2+ \frac{r}{m}-\frac{u}{m}) \ge 1-c(2+ \frac{r}{m}) $$ 
for $m$ large. 
These imply that $f_m(m^{\frac{c}{m}})<g_m(m^{\frac{c}{m}}) $ holds for $m $ large, 
since $m^{1 -c(1+ \frac{r}{m}- \frac{j}{m})}$ is appeared in the numerator of $g_m(m^{\frac{c}{m}})$, 
see (\ref{equation_g_m}).

Next, we suppose that $\frac{1}{2} \le  c <1$. 
We can check that 
$m^{1 -c(1+ \frac{r}{m}- \frac{j}{m})}$ is still dominant among the four in (\ref{equation_four}). 
(The second and fourth terms are bounded as $m$ goes to $\infty$.) 
Therefore we still have $f_m(m^{\frac{c}{m}})<g_m(m^{\frac{c}{m}}) $  for $m $ large. 

Finally we suppose that  $c>1$. 
Then the last three terms in (\ref{equation_four})  go to $0$ as $m$ goes to $\infty$, 
because they have the negative powers of $m$ for $m$ large.  
Thus the numerator of $g_m(m^{\frac{c}{m}})$, see (\ref{equation_g_m}), goes to $0$ as $m$ tends to $\infty$.  
On the other hand, 
$f_m(m^{\frac{c}{m}}) > \frac{sc \log m}{m} $ holds (see (\ref{equation_second})),  and 
hence the numerator of 
$$\frac{sc \log m + mR_2+ m^{1-c(2+ \frac{r}{m})}}{m} (= f_m(m^{\frac{c}{m}}) )$$ 
goes to $\infty$ as $m$ does.  
Thus $f_m(m^{\frac{c}{m}})>g_m(m^{\frac{c}{m}}) $ for $m $ large. 
This completes the proof of the first part of the claim (2). 

Taking the logarithm of the both sides  of 
$m^{\frac{c_1}{m}} < \lambda_m<  m^{\frac{c_2}{m}}$ 
yields 
$$c_1 < \frac{m \log \lambda_m}{ \log m} < c_2 \hspace{2mm}\mbox{for\ }m\ \mbox{large}.$$ 
Since $0<c_1<1$ and $c_2>1$ are any numbers, 
we have the desired limit.  
This completes the proof of the second half of the claim (2). 
\medskip

\noindent
(3) 
By the claim (2), 
$$\frac{c_1 \log m}{m} < \log \lambda_m < \frac{c_2 \log m}{m} \hspace{2mm}\mbox{for\ }m\ \mbox{large}.$$ 
Let us set $n= qm+v$. 
We substitute $m= \frac{n-v}{q}$ for the inequality above: 
$$\frac{c_1 \log \Bigl(\frac{n-v}{q} \Bigr)}{  \frac{n-v}{q}}<  \log \lambda_m<  
\frac{c_2 \log\Bigl(\frac{n-v}{q} \Bigr)}{  \frac{n-v}{q}}.$$
Hence 
$$\frac{q c_1 \bigl(\log(n-v) - \log q \bigr)}{n-v} < \log \lambda_m<  \frac{q c_2 \bigl(\log(n-v) - \log q \bigr)}{n-v}  .$$
We multiply all sides above by $\frac{n}{\log n}>0$ (for $n$ large). 
Then 
$$\frac{q c_1 n \bigl(\log(n-v) - \log q \bigr)}{ (n-v)\log n } 
< \frac{n \log \lambda_m}{ \log n}
<  \frac{q c_2 n \bigl(\log(n-v) - \log q \bigr)}{  (n-v)\log n} .$$
Note that 
$\frac{ n \bigl(\log(n-v) - \log q \bigr)}{ (n-v)\log n } $ goes to $1$ 
as $n$ (and hence $m$) goes to $\infty$.  
Since $0 < c_1 < 1$ and $c_2>1$ are any numbers, 
it follows that 
$$\lim_{m \to \infty} \frac{n \log \lambda_m}{ \log n}=\lim_{m \to \infty} \frac{(qm+v) \log \lambda_m}{ \log (qm+v)}=q.$$ 
\end{proof}

\section{Thurston norm and fibered $3$-manifolds}
\label{section_Thurston}

Let $M$ be an  oriented hyperbolic $3$-manifold 
with boundary $\partial M$ (possibly $\partial M = \emptyset$). 
The Thurston norm 
$\| \cdot \|$  is defined on 
an integral class  $a \in H_2(M, \partial M; {\Bbb Z})$  as follows. 
\begin{equation*} 
	\| a\|= \min_F \{- \chi(F)\}, 
\end{equation*} 
where the minimum is taken over 
all oriented surface  $F$  embedded in $M$, satisfying  $a= [F]$,  
with no components  of non-negative Euler characteristic. 
The surface  $F$  which realizes this minimum is called 
the {\it minimal representative} of $a $, denoted by  $F_a$. 
The norm  $\| \cdot \|$  defined on integral  classes 
admits a unique continuous extension 
$\| \cdot \|: H_2(M, \partial M; {\Bbb R})  \rightarrow {\Bbb R}$  which is 
linear on the ray through the origin. 
The unit ball  $U_M$ with respect to the Thurston norm  is 
a compact, convex polyhedron. 
See  \cite{Thurston1} for more details.

Suppose that $M$ is a surface bundle over the circle and 
let  $F$  be its fiber. 
The fibration determines a cohomology 
class  $a^* \in H^1(M; \mathbb{Z})$,   
and hence a homology class  $a \in H_2(M, \partial M; \mathbb{Z})$  
by Poincar\'e duality.  
Thurston proved in \cite{Thurston1} that 
there exists a top dimensional face  $\Omega$ on $\partial U_M$  such that 
$a = [F]$  is an integral class of  $int(C_{\Omega})$.  
On the other hand,  
the minimal representative  $F_a$  
for any integral class  $a$  in  $int(C_{\Omega})$   
becomes a fiber of the fibration associated to  $a$.  
Such a face $\Omega$  is called a {\it fibered face},  
and an integral class  $a \in int (C_{\Omega})$  is called a {\it fibered class}.  

The set of integral and rational classes 
of  $int(C_{\Omega})$  are denoted by 
$int(C_{\Omega}({\Bbb Z}))$  and  $int(C_{\Omega}({\Bbb Q}))$  respectively. 
When  $a \in int(C_{\Omega}(\mathbb{Z}))$  is primitive, 
the associated fibration on  $M$  has a connected fiber  
represented by $F_a$.  
Let  $\Phi_a : F_a \rightarrow F_a$  be the monodromy. 
Since  $M$  is hyperbolic,  
$\phi_a = [\Phi_a]$  is pseudo-Anosov. 
The {\it dilatation}  $\lambda(a)$  and  {\it entropy}  
$\mathrm{ent}(a) = \log \lambda(a)$  are defined as the dilatation 
and entropy of $\phi_a$ respectively.

The entropy defined on primitive fibered classes is extended to 
rational classes as follows:
For a rational number $r $ and a primitive fibered class $a$, 
the entropy  $\mathrm{ent}(ra)$ is defined 
by  $\frac{1} {|r|}  \mathrm{ent}(a)$.  
It is shown by Fried in \cite{Fried} that 
$\frac{1}{\mathrm{ent}}: int(C_{\Omega}({\Bbb Q})) \rightarrow {\Bbb R}$ 
is concave, 
and in particular 
$\mathrm{ent}: int(C_{\Omega}({\Bbb Q})) \rightarrow {\Bbb R}$  
admits a unique continuous extension 
\begin{equation*} 
	\mathrm{ent}: int(C_{\Omega}) \rightarrow {\Bbb R}.  
\end{equation*}  
Moreover 
Fried proved that 
 the restriction of  $\mathrm{ent}$  to 
the open fibered face  $int(\Omega)$ has the property such that 
$\mathrm{ent}(a)$  goes to  $ \infty$  as  $a \in int(\Omega)$  goes to 
a point on  $\partial \Omega$.  
Thus we have a continuous function 
\begin{equation*} 
	\mathrm{Ent}= \|\cdot \| \, \mathrm{ent}(\cdot) : int(C_{\Omega}) \to \mathbb{R}
\end{equation*} 
which is constant on each ray in  $int(C_{\Omega})$  through the origin. 
Thus the function 
$$\mathrm{Ent}|_{int(\Omega)}(= \mathrm{ent}|_{int(\Omega)}): int(\Omega) \rightarrow {\Bbb R}$$ 
has a minimum, denoted by $\min \mathrm{Ent}(M, \Omega)$. 
Matsumoto\cite{Matsumoto} refined the result by Fried. 
(See also McMullen\cite{McMullen}.) 
He proved that 
$\frac{1}{\mathrm{ent}}|_{int(\Omega)}: int(\Omega) \rightarrow {\Bbb R}$ 
is strictly concave. 
This implies that 
$\min \mathrm{Ent}(M, \Omega)$ is achieved by a unique point in $int(\Omega)$. 
The quantity $\min \mathrm{Ent}(M, \Omega)$ is a significant invariant on the pairs $(M, \Omega)$, 
but we do not discuss this invariant in the present paper. 

Teichm\"{u}ller polynomial $P_{\Omega}$, developed by McMullen\cite{McMullen} 
organizes the dilatations $\lambda(a)$ for all $a \in int(C_{\Omega})$. 
Once one computes $P_{\Omega}$, 
the largest real root of the polynomial determined by $P_{\Omega}$ and a given fibered class $a \in int(C_{\Omega})$ 
gives us the dilatation $\lambda(a)$.

\section{The magic $3$-manifold $N$}
\label{section_magic}

Monodromies of fibrations on $N$  have been studied in \cite{KKT,KT,KT1}. 
(See also a survey \cite{Kin}.) 
In Sections~\ref{subsection_Fibered} and \ref{subsection_DilatatioFiber}, 
we recall some results which tell us that the topology of fibered classes $a$ and the actual value of  $\lambda(a)$. 
In Section~\ref{subsection_FiberedClassMain}, 
we define a family of fibered classes $a_{(g,p)}$ of $N$ with two variables $g$ and $p$, 
and we shall prove that it is a suitable family to prove theorems in Section~\ref{section_intro} (cf. Remark~\ref{rem_experiment}).

Recall that $\varSigma_{g,n}$  is an orientable surface of genus $g$ with $n$ punctures. 
Abusing the notation, we sometimes denote by $\varSigma_{g,n}$, 
an orientable surface of genus $g$ with $n$ boundary components.

\subsection{Fibered face $\Delta$}
\label{subsection_Fibered}

Let $K_{\alpha}$, $K_{\beta}$ and  $K_{\gamma}$ be the components of  the $3$ chain link $\mathcal{C}_3$. 
They bound the oriented disks $F_{\alpha}$, $F_{\beta}$ and $F_{\gamma}$ with $2$ holes, see Figure~\ref{fig_poly}. 
Let $\alpha = [F_{\alpha}]$, $ \beta = [F_{\beta}]$, $\gamma= [F_{\gamma}] \in H_2(N, \partial N;{\Bbb Z})$. 
The set $\{\alpha, \beta, \gamma\}$ is a basis of $H_2(N,  \partial N; {\Bbb Z})$. 
Figure~\ref{fig_poly} illustrates 
the Thurston norm  ball $U_N$ for $N$ which  is the parallelepiped  with vertices 
$\pm \alpha $, $\pm \beta $, $\pm \gamma$, $\pm(\alpha + \beta + \gamma)$ (\cite[Example~3 in Section~2]{Thurston1}).  
Because of the symmetry of $\mathcal{C}_3$, 
every top dimensional face of $U_N$ is a fibered face. 

We denote a class $x \alpha + y \beta + z \gamma \in H_2(N, \partial N)$ by $(x,y,z)$. 
We pick a  fibered  face $\Delta$
with vertices $\alpha = (1,0,0)$, $\alpha+ \beta + \gamma= (1,1,1)$, $\beta=(0,1,0)$ and $-\gamma= (0,0,-1)$, 
see Figure~\ref{fig_poly}.  
The open face $ int(\Delta) $ is written by 
$$ int(\Delta) = \{(X,Y,Z)\ |\ X+Y-Z =1, \ X >0,\  Y>0,\   X >Z,\   Y>Z\}.$$  
A class  $a= (x,y,z) \in H_2(N, \partial N)$ is an element of $int(C_{\Delta})$ if and only if 
$x >0$, $y>0$, $x >z$ and $y>z$. 
In this case, we have $\|a\|= x+y-z$.

Let  $ a= (x,y,z)$ be a  fibered class in $int(C_{\Delta})$. 
The minimal representative of this class is denoted by  $F_{a}$ or $F_{(x,y,z)}$. 
We  recall the formula which tells us that the number of the boundary components of $F_a$. 
We denote the tori $ \partial \mathcal{N}(K_{\alpha})$, $ \partial \mathcal{N}(K_{\beta})$, $ \partial \mathcal{N}(K_{\gamma})$ by 
$T_{\alpha}$, $T_{\beta}$, $T_{\gamma}$ respectively, 
where  $\mathcal{N}(K)$ be a  regular neighborhood of a knot $K$ in $S^3$. 
Let us set 
$\partial_{\alpha} F_{(x,y,z)} = \partial F_{(x,y,z)} \cap T_{\alpha}$ which consists of  the parallel simple closed curves on $T_{\alpha}$. 
We define the subsets $\partial_{\beta} F_{(x,y,z)} $, $\partial_{\gamma} F_{(x,y,z)} \subset  \partial F_{(x,y,z)} $ in the same manner. 
By \cite[Lemma~3.1]{KT}, 
the number of the boundary components 
$\sharp(\partial F_{(x,y,z)}) = 
\sharp( \partial_{\alpha} F_{(x,y,z)}) + \sharp (\partial_{\beta} F_{(x,y,z)} )+ \sharp( \partial_{\gamma} F_{(x,y,z)})$ is given by 
\begin{equation}
\label{equation_num_boundary}
\sharp( \partial_{\alpha} F_{(x,y,z)}) = \gcd(x,y+z),\ \sharp (\partial_{\beta} F_{(x,y,z)} )=  \gcd(y,z+x),\ \sharp( \partial_{\gamma} F_{(x,y,z)}) = \gcd(z,x+y),
\end{equation}
where $\gcd(0,w)$ is defined by $|w|$.

\subsection{Dilatations $\lambda(a)$ and the stable foliation $\mathcal{F}_a$ of fibered classes $a$}
\label{subsection_DilatatioFiber}

Teichm\"{u}ller polynomial $P_{\Delta}$ on the fibered face $\Delta$ is computed in 
\cite[Section~3.2]{KT}, and it tells us that 
the dilatation $\lambda_{(x,y,z)}$ of a fibered class $(x,y,z) \in int(C_{\Delta})$ 
is the largest real root of 
$$f_{(x,y,z)}(t)= t^{x+y-z}-t^x - t^y - t^{x-z}- t^{y-z}+1,$$ 
see \cite[Theorem~3.1]{KT}. 
(In fact, $\lambda_{(x,y,z)}$ is a unique real root greater than $1$ of $f_{(x,y,z)}(t)$ 
by Descartes's rule of signs.)

Let $\Phi_{(x,y,z)}: F_{(x,y,z)} \rightarrow  F_{(x,y,z)}$ 
be the monodromy of the fibration associated to  a primitive  class $(x,y,z) \in int(C_{\Delta})$. 
Let $\mathcal{F}_{(x,y,z)}$ be the stable foliation of the pseudo-Anosov $\Phi_{(x,y,z)}$. 
The components of $\partial_{\alpha} F_{(x,y,z)}$ (resp. $\partial_{\beta} F_{(x,y,z)}$, $\partial_{\gamma} F_{(x,y,z)}$) 
are permuted cyclically by $\Phi_{(x,y,z)}$. 
In particular the number of prongs of $\mathcal{F}_{(x,y,z)}$ at a component of 
$\partial_{\alpha} F_{(x,y,z)}$ (resp. $\partial_{\beta} F_{(x,y,z)}$, $\partial_{\gamma} F_{(x,y,z)}$) 
is independent of the choice of the component. 
By \cite[Proposition~3.3]{KT1},  the stable foliation  $\mathcal{F}_{(x,y,z)}$  has the property such that: 
\begin{itemize}
\item 
each component  of $\partial_{\alpha} F_{(x,y,z)}$ has 
$x/\gcd(x,y+z)$ prongs, 

\item 
each component  of  $\partial_{\beta} F_{(x,y,z)}$ has 
$y/\gcd(y,x+z)$ prongs, and 

\item 
each component of $\partial_{\gamma} F_{(x,y,z)}$ has 
$(x+y-2z)/\gcd(z,x+y)$ prongs. 

\item 
$\mathcal{F}_{(x,y,z)}$ does not have singularities in the interior of $F_{(x,y,z)} $. 
\end{itemize}

\subsection{Proofs of theorems}
\label{subsection_FiberedClassMain}

For $g \ge 0$ and  $p \ge 0$, 
define a fibered class $a_{(g,p)}$ as follows. 
$$a_{(g,p)}= (p+g+1) \mathfrak{a} + (p-g)  \mathfrak{b}= 
(p+g+1, 2p+1, p-g) \in  int(C_{\Delta}).$$ 
The class $a_{(g,p)}$ is  primitive if and only if $2g+1$ and $p+g+1$ are relatively prime. 
One can check the identity 
$$B_{(g,p)}(t) =  f_{(p+g+1, 2p+1, p-g)}(t)$$  
(see Section~\ref{section_intro} for the definition of $B_{(g,p)}(t)$). 
We denote by  $r_{(g,p)}$, 
the dilatation $ \lambda(a_{(g,p)})$ of the fibered class $a_{g,p}$. 
(Thus the dilatation $r_{(g,p)} = \lambda(a_{(g,p)})$ of $a_{(g,p)}$ is a unique real root greater than $1$ of $B_{(g,p)}(t) $, 
see Section~\ref{subsection_DilatatioFiber}.)

\begin{lem}
\label{lem_asymp_agp}
We fix $g \ge 0$.  
Given $0 < c_1 <1$ and  $c_2>1$, 
we have 
$$p^{\frac{c_1}{p}} < r_{(g,p)}< p^{\frac{c_2}{p}}  \hspace{2mm} \mbox{for\  } p\ \mbox{large}. $$
In particular 
$$ \displaystyle\lim_{p \to \infty} \frac{p \log r_{(g,p)}}{ \log p}=1.$$
\end{lem}

\begin{proof}
Apply  Lemma~\ref{lem_asymptoticPoly} to the polynomial $B_{(g,p)}(t)$. 
\end{proof}


\begin{lem}
\label{lem_data_agp}
Suppose that  $a_{(g,p)}$ is  primitive. 
Then the minimal representative $F_{a_{(g,p)}}$ is a surface of genus $g$ with $2p+4$ boundary components, and 
the stable foliation $\mathcal{F}_{a_{(g,p)}}$ has the following properties. 
If $p+g$ is odd (resp. even), then 
$\sharp(\partial_{\alpha} F_{a_{(g,p)}})= 2$ (resp. $1$)  and 
$\sharp(\partial_{\gamma} F_{a_{(g,p)}})= 1$ (resp. $2$). 
 A component of $\partial_{\alpha}  F_{a_{(g,p)}}$ has $\frac{p+g+1}{2}$ prongs (resp. $(p+g+1)$ prongs), and 
  a component of $\partial_{\gamma}  F_{a_{(g,p)}}$ has  $(p+3g+2)$ prongs (resp. $\frac{p+3g+2}{2}$ prongs). 
%
\end{lem}

\begin{proof} 
By (\ref{equation_num_boundary}), 
we have  that $\sharp(\partial_{\beta} F_{a_{(g,p)}})= 2p+1 $. 
We have 
$$\sharp(\partial_{\alpha} F_{a_{(g,p)}})= \gcd(p+g+1, 3p-g+1) = \gcd(p+g+1, 2(2g+1)).$$ 
Since $a_{(g,p)}$ is primitive, $p+g+1$ and $2g+1$ must be  relatively prime. 
Hence 
$\sharp(\partial_{\alpha} F_{a_{(g,p)}})=1$ (resp. $2$) 
if $p+g$ is even (resp. odd). 
Let us compute $\sharp(\partial_{\gamma} F_{a_{(g,p)}})$. 
We have 
$$\sharp(\partial_{\gamma} F_{a_{(g,p)}})= \gcd(3p+g+2, p-g)= \gcd(2(2g+1), p-g).$$ 
Since $\gcd(2g+1, p-g)= \gcd(2g+1, p+g+1)=1$, 
we have that 
$\sharp(\partial_{\gamma} F_{a_{(g,p)}})= 2$ (resp. $1$) if $p-g$ is even (resp. odd), equivalently $p+g$ is even (resp. odd).  
The genus of $F_{a_{(g,p)}}$ is computed from the identities 
$\|a_{(g,p)}\|(= |\chi(F_{a_{(g,p)}})|) = 2p+2g+2$ and $\sharp(\partial F_{a_{(g,p)}})= 2p+4$.

The singularity data of $\mathcal{F}_{a_{(g,p)}}$ is obtained from the formula in Section~\ref{subsection_DilatatioFiber}. 
\end{proof}

\noindent
By Lemma~\ref{lem_data_agp}, it is straightforward to prove

\begin{lem}
\label{lem_No_one}
Suppose that $a_{(g,p)}$ is primitive. 
Then 
$(g,p) \not\in\{(0,0), (0,1), (1,0)\}$ 
if and only if 
$\mathcal{F}_{a_{(g,p)}}$ has the property such that 
each component of $\partial_{\alpha}  F_{a_{(g,p)}} \cup \partial_{\gamma}  F_{a_{(g,p)}}$
does not have $1$ prong. 
In particular if $g \ge 2$ and $p \ge 0$, then 
each component of $\partial_{\alpha}  F_{a_{(g,p)}} \cup \partial_{\gamma}  F_{a_{(g,p)}}$ 
does not have $1$ prong. 
\end{lem}

We are now ready to prove theorems in Section~\ref{section_intro}. 

\begin{proof}[Proof of Theorem~\ref{thm_Asymptotic}]
There exists a sequence of primitive fibered classes $\{a_{(g, p_i)}\}_{i=0}^{\infty}$ with $p_i \to \infty$. 
(In fact, if we take $p_i= (g+1)+ (2g+1)i$, then 
$2g+1$ and $p_i+g+1$ are relatively prime. 
Hence $a_{(g, p_i)}$ is primitive.) 
Then $N$ is a $\varSigma_{g, 2p_i+4}$-bundle over the circle 
whose monodromy of the fibration has the dilatation $r_{(g,p_i)}$. 
Therefore $\delta_{g,2p_i+4} \le r_{(g,p_i)}$. 
If we set $n_i= 2p_i+4$, then 
$$\frac{n_i \log \delta_{g, n_i}}{\log n_i} \le \frac{n_i \log r_{(g,p_i)}}{ \log n_i}=  \frac{(2p_i+4) r_{(g,p_i)}}{ \log (2p_i+4)}.$$
The right hand side goes to $2$ as $i$ goes to $\infty$, see 
Lemmas~\ref{lem_asymptoticPoly}(3) and \ref{lem_asymp_agp}. 
This completes the proof. 
\end{proof}

\begin{proof}[Proof of Theorem~\ref{thm_UpperBound}.]
The monodromy $\Phi_{a_{(g,p)}}$ of the fibration associated to the primitive fibered class $a_{(g,p)}$ 
is defined on the surface of genus $g$ with $2p+4$ boundary components. 
It  has the dilatation $r_{(g,p)}$, and hence 
$\delta_{g, 2p+4} \le r_{(g,p)}$. 

Now let us prove $\delta_{g, 2p+1} \le r_{(g,p)}$. 
The fibration associated to $a_{(g,p)}$ extends naturally to a fibration 
on the manifold obtained from $N$ by Dehn filling two cusps 
specified by the tori $T_{\alpha}$ and $T_{\gamma}$ 
along the boundary slopes of the fiber.  
Then  $\Phi_{a_{(g,p)}}: F_{a_{(g,p)}} \rightarrow F_{a_{(g,p)}} $ 
extends to the monodromy $\widehat{\Phi}: \widehat{F} \rightarrow \widehat{F}$ 
of the extended fibration, 
where the extended fiber $\widehat{F}$ is obtained from $F_{a_{(g,p)}}$ by filling each disk bounded by 
each component of $\partial_{\alpha}  F_{a_{(g,p)}} \cup \partial_{\gamma}  F_{a_{(g,p)}}$. 
Thus $\widehat{F}$ has the genus $g$ with $2p+1$ boundary components, see Lemma~\ref{lem_data_agp}. 
By Lemma~\ref{lem_No_one}, 
 $\mathcal{F}_{a_{(g,p)}}$ does not have $1$ prong at each component of 
$\partial_{\alpha}  F_{a_{(g,p)}} \cup \partial_{\gamma}  F_{a_{(g,p)}}$. 
Hence $\mathcal{F}_{a_{(g,p)}}$ extends canonically to the stable foliation $\widehat{\mathcal{F}}$ of $\widehat{\Phi}$. 
Therefore $\widehat{\phi}= [\widehat{\Phi}] $ is pseudo-Anosov with the same dilatation as $\Phi_{a_{(g,p)}}$. 
This implies that 
$\delta_{g, 2p+1} \le r_{(g,p)}$. 

The proofs of the rest of bounds $\delta_{g, 2p+2} \le r_{(g,p)}$ and $\delta_{g, 2p+3} \le r_{(g,p)}$ are similar. 
In fact, the extended fiber of the fibration on the manifold obtained from $N$ 
by Dehn filling a cusp specified by $T_{\alpha}$ or $T_{\gamma}$ 
along the boundary slope of the fiber  
has the genus $g$ with $2p+2$ or $2p+3$ boundary components, see Lemma~\ref{lem_data_agp}. 
Lemma~\ref{lem_No_one} ensures that the extended monodromy is pseudo-Anosov with the same dilatation as $\Phi_{a_{(g,p)}}$. 
\end{proof}

\begin{proof}[Proof of Theorem~\ref{thm_main}]
By Theorem~\ref{thm_UpperBound} together with the assumption $(*)$ in  Theorem~\ref{thm_main}, 
we have that for any $p \ge 0$ and for $j \in \{3,4\}$, 
$$\delta_{g,2p+j} \le r_{(g,p)} \hspace{2mm}\mbox{or} \hspace{2mm}\delta_{g,2p+j}\le r_{(g,p+1)}.$$
Thus 
\begin{equation}
\label{equation_claim_main}
\frac{(2p+j) \log \delta_{g, 2p+j}}{\log(2p+j)} \le  \frac{(2p+j) \log r_{(g,p)}}{\log(2p+j)} \hspace{2mm} 
\mbox{or} \hspace{2mm} 
\frac{(2p+j) \log \delta_{g, 2p+j}}{\log(2p+j)} \le  \frac{(2p+j) \log r_{(g,p+1)}}{\log(2p+j)}.
\end{equation}
By Lemma~\ref{lem_asymptoticPoly}, 
it is easy to see that 
the both right hand sides in (\ref{equation_claim_main}) 
go to $2$ as $p$ goes to $\infty$. 
Thus 
$$\limsup_{p \to \infty} \frac{(2p+j) \log \delta_{g, 2p+j}}{\log(2p+j)} \le 2.$$ 
Since this holds for  $j \in \{3,4\}$, the proof is done. 
\end{proof}

\begin{proof}[{\it Proof of Proposition~\ref{prop_volume}}]
We prove the claim in the second half. 
(The proof in the first half is similar.) 
If $g \ge 2$ satisfies $(*)$, then for any $p \ge 0$ 
there exist a $\varSigma_{g, 2p+3}$-bundle and a $\varSigma_{g, 2p+4}$-bundle over the circle 
obtained from $N$, 
see proof of Theorem~\ref{thm_main}. 
More precisely 
such a bundle is homeomorphic to $N$ or 
it is obtained from $N$ by Dehn filling cusps along the boundary slopes of the fiber. 
Thus Proposition~\ref{prop_volume} holds from the result which says that 
the hyperbolic volume decreases after Dehn filling, see \cite{NZ,Thurston}. 
\end{proof}

\begin{rem}
\label{rem_experiment}
To address Question~\ref{ques_analogy}, 
we had explored fibered classes of the magic manifold 
whose dilatations have a suitable asymptotic behavior. 
We found a family of primitive fibered classes $a_{(g,p)}$ by computer. 
By Lemma~\ref{lem_data_agp}, 
most of the components of $\partial F_{a_{(g,p)}}$ lie on the torus $T_{\beta}$. 
The pseudo-Anosov stable foliation associated to $a_{(g,p)}$ has the property such that 
each  component of $\partial_{\beta} F_{a_{(g,p)}}$ has $1$ prong. 
The striking property of $a_{(g,p)}$ is  that 
the slope of the components of $\partial_{\beta} F_{a_{(g,p)}}$ is exactly equal to $-1$. 
Moreover, for any fixed $g$,  the projective class $\overline{a}_{(g,p)}$ goes to a single point $(\tfrac{1}{2}, 1, \tfrac{1}{2}) \in \partial \Delta$ 
as $p$ goes to $\infty$. 
On the other hand, it is proved by Martelli and Petronio\cite{MP} that 
the manifold $N(-1)$ obtained from $N$ by Dehn filling a cusp along the boundary slope $-1$ is not hyperbolic. 
The property such that each component of $\partial_{\beta} F_{a_{(g,p)}}$ has $1$ prong 
can also be seen from the fact hat $N(-1)$ is a non hyperbolic manifold.  
\end{rem}


\begin{thebibliography}{99}

\bibitem{AD} 
J.~W.~Aaber and N.~M.~Dunfield, 
{\it Closed surface bundles of least volume}, 
Algebraic and Geometric Topology 10 (2010), 2315-2342. 









\bibitem{CH} 
J.~H.~Cho and J.~Y.~Ham,   
{\it The minimal dilatation of a genus-two surface}, 
Experimental Mathematics 17 (2008), 257-267. 





\bibitem{FLP}
A.~Fathi, F.~Laudenbach and V.~Poenaru, 
Travaux de Thurston sur les surfaces,
Ast\'erisque, 66-67, 
Soci\'et\'e Math\'ematique de France, Paris (1979). 

\bibitem{Fried}
D.~Fried, 
{\it Flow equivalence, hyperbolic systems and a new zeta function for flows}, 
Commentarii Mathematici Helvetici 57 (1982), 237-259. 




\bibitem{Hironaka} 
E.~Hironaka, 
{\it Small dilatation  mapping classes coming from the simplest hyperbolic braid}, 
Algebraic and Geometric Topology 10 (2010), 2041-2060. 

\bibitem{HK}
E.~Hironaka and E.~Kin, 
{\it A family of pseudo-Anosov braids with small dilatation}, 
Algebraic and Geometric Topology 6 (2006), 699-738. 

\bibitem{Ivanov}
N.~V.~Ivanov,
{\it Stretching factors of pseudo-Anosov homeomorphisms},
Journal of Soviet Mathematics, 52 (1990), 2819--2822, which is translated from
Zap. Nauchu. Sem. Leningrad. Otdel. Mat. Inst. Steklov.
(LOMI), 167 (1988), 111--116.

\bibitem{Kin} 
E.~Kin, 
{\it Notes on pseudo-Anosovs with small dilatations coming from the magic 3-manifold}, 
Representation spaces, twisted topological invariants and geometric structures of 3-manifolds, 
RIMS Kokyuroku 1836 (2013) 45-64. 


\bibitem{KKT} 
E.~Kin, S.~Kojima and M.~Takasawa, 
{\it Minimal dilatations of pseudo-Anosovs generated by the magic $3$-manifold and their asymptotic behavior}, 
preprint (2011), 
arXiv:1104.3939v3, 
to appear in ``Algebraic and geometric topology". 


%
 

\bibitem{KT} 
E.~Kin and M.~Takasawa, 
{\it Pseudo-Anosov braids with small entropy and the magic $3$-manifold}, 
Communications in Analysis and Geometry 19 (4) (2011), 705-758. 

\bibitem{KT1} 
E.~Kin and M.~Takasawa, 
{\it Pseudo-Anosovs on closed surfaces having small entropy and the Whitehead sister link exterior}, 
Journal of the Mathematical Society of Japan 65 (2) (2013), 411-446.  




\bibitem{MP} 
B.~Martelli and C.~Petronio, 
{\it Dehn filling of the ``magic" $3$-manifold}, 
Communications in Analysis and Geometry 14 (2006), 969-1026.

\bibitem{Matsumoto}
S.~Matsumoto, 
{\it Topological entropy and Thurston's norm of atoroidal surface bundles over the circle}, 
Journal of the Faculty of Science, University of Tokyo, Section IA. Mathematics 34 (1987), 763-778. 





\bibitem{McMullen}
C.~McMullen, 
{\it Polynomial invariants for fibered $3$-manifolds and Teichm\"{u}ler geodesic for foliations}, 
Annales Scientifiques de l'\'{E}cole Normale Sup\'{e}rieure. Quatri\`{e}me S\'{e}rie  33 (2000), 519-560. 


\bibitem{NZ} 
W.~D.~Neumann and D.~Zagier, 
Volumes of hyperbolic three-manifolds, 
Topology 24 (3) (1985), 
307-332. 


\bibitem{Penner}
R.~C.~Penner, 
{\it Bounds on least dilatations}, 
Proceedings of the American Mathematical Society 113 (1991), 443-450. 

\bibitem{Thurston}
W.~Thurston, 
The geometry and topology of $3$-manifolds, 
Lecture Notes, Princeton University (1979). 

\bibitem{Thurston1}
W.~Thurston, 
{\it A norm of the homology of $3$-manifolds}, 
Memoirs of the American Mathematical Society 339 (1986), 99-130. 

\bibitem{Thurston2} 
W.~Thurston, 
{\it On the geometry and dynamics of diffeomorphisms of surfaces}, 
Bulletin of the American Mathematical Society 19 (1988),  417-431.

\bibitem{Thurston3} 
W.~Thurston, 
{\it Hyperbolic structures on 3-manifolds II: Surface groups and 
3-manifolds which fiber over the circle}, preprint, 
arXiv:math/9801045

\bibitem{Tsai} 
C.~Y.~Tsai, 
{\it The asymptotic behavior of least pseudo-Anosov dilatations}, 
Geometry and Topology 13 (2009), 2253-2278. 

%
\bibitem{Valdivia} 
A.~D.~Valdivia, 
{\it Sequences of pseudo-Anosov mapping classes and their asymptotic behavior}, 
New York Journal of Mathematics 18 (2012), 609-620.


\bibitem{Wang} 
X.~Wang, 
{\it A simple proof of Descartes's rule of signs}, 
American Mathematical Monthly
111 (6) (2004),  525-526.  







\end{thebibliography}
\end{document}